\documentclass[11pt,a4paper,reqno]{amsart}
\usepackage{amsthm}
\usepackage{amsmath}
\usepackage{amssymb}
\usepackage{graphicx,color}

\usepackage[font=small]{caption}

\usepackage[shortlabels]{enumitem}

\makeatletter
\theoremstyle{plain}
\newtheorem{thm}{Theorem}
  \theoremstyle{definition}
  
  \theoremstyle{remark}
  \newtheorem{rem}[thm]{Remark}
  \theoremstyle{plain}
  
  \theoremstyle{plain}
  
  \theoremstyle{plain}
  \newtheorem{cor}[thm]{Corollary}
 \theoremstyle{definition}
  
  \theoremstyle{remark}
  \newtheorem*{rem*}{Remark}

  \theoremstyle{definition}

\usepackage{amsfonts}
\usepackage{mathrsfs}

\newcommand{\N}{\mathbb{N}}
\newcommand{\R}{{\mathbb{R}}}
\newcommand{\C}{{\mathbb{C}}}
\newcommand{\Z}{{\mathbb{Z}}}
\newcommand{\ii}{{\rm i}}

\newcommand{\diag}{\mathop\mathrm{diag}\nolimits}

\renewcommand{\Re}{\mathop\mathrm{Re}\nolimits}
\renewcommand{\Im}{\mathop\mathrm{Im}\nolimits}

\makeatother

\begin{document}

\title[]{The Hurwitz-type theorem for the regular Coulomb wave function via Hankel determinants}

\author[\'A. Baricz]{\'Arp\'ad Baricz}

\address{Department of Economics, Babe\c{s}-Bolyai University, Cluj-Napoca, Romania}
\address{Institute of Applied Mathematics, \'Obuda University, Budapest, Hungary}
\email{bariczocsi@yahoo.com}

\author[F. \v{S}tampach]{Franti\v{s}ek \v{S}tampach}

\address{Department of Mathematics, Stockholm University, Stockholm, Sweden}
\address{Department of Applied Mathematics, Faculty of Information Technology, Czech Technical University in~Prague, 
	Th{\' a}kurova~9, 160~00 Praha, Czech Republic}
\email{stampfra@fit.cvut.cz}

\keywords{Hankel determinant, Coulomb wave function, Bessel function, Rayleigh function}

\subjclass[2010]{15A15, 33C15, 33C10}

\begin{abstract}
We derive a closed formula for the determinant of the Hankel matrix whose entries are given by sums of negative powers of the zeros of the regular Coulomb wave function.
This new identity applied together with results of Grommer and Chebotarev allows us to prove a Hurwitz-type theorem about the zeros of the regular 
Coulomb wave function. As a particular case, we obtain a new proof of the classical Hurwitz's theorem from the theory of Bessel functions that is based on algebraic 
arguments. In addition, several Hankel determinants with entries given by the Rayleigh function and Bernoulli numbers are also evaluated.
\end{abstract}

\maketitle

\section{Introduction}

The Bessel functions as well as the Coulomb wave functions belong to the very classical special functions
that appear frequently at various places in mathematics and physics. 
Since the regular Coulomb wave function represents a generalization of the Bessel function of the first kind,
its properties are often reminiscent to the respective properties of the Bessel function. On the other hand, 
not all methods applicable to Bessel functions admit a straightforward generalization to the case of Coulomb wave functions.

A significant part of the research on Bessel functions and their generalizations is devoted to the study of their zeros. 
Naturally, methods of mathematical, in particular, complex analysis turn out to be the very well suited techniques in 
the analysis of the zeros that prove their usefulness many times during the last century. However, one of the main aim of this article 
is to stress the importance of linear algebra in the analysis of zeros of entire functions. By using primarily linear algebraic techniques, 
we prove a theorem on the reality and the exact number of possible complex zeros of the regular Coulomb wave function. The proof consists 
of two main ingredients. First, with the aid of certain properties of a particularly chosen family of orthogonal polynomials which is intimately 
related to the Coulomb wave function and was studied in~\cite{stampachstovicek_jmaa14}, we evaluate the determinant of the Hankel matrix whose 
entries are given by sums of negative powers of the zeros. Second, we combine the formula for the Hankel determinant with general results of 
Grommer~\cite{grommer_14} and Chebotarev~\cite{chebotarev_ma28}, which straightforwardly yields the desired goal.

The obtained result generalizes the well-known theorem of Hurwitz about the zeros of the Bessel function of the first kind; see Theorem~\ref{thm:hurwitz} below 
for its formulation. Consequently, the present approach provides an alternative proof of this classical result, which is rather algebraic and simple.
Let us remark that the original proof from Hurwitz~\cite{hurwitz_ma88}, based on the connection between Lommel polynomials and Bessel functions, 
need not be easy to read nowadays since the modern terms were not introduced in his time. In addition, the original proof contained some gaps that were corrected 
later by Watson~\cite{watson_plms20}. Nonetheless, the Hurwitz's result was very influential and many other proofs were found; see, for example, the articles of
Hilb~\cite{hilb_mz22}, Obreschkoff~\cite{obreschkoff_jdmv29}, P{\' o}lya~\cite{polya_jdmv29}, Hille and Szeg\H{o}~\cite{hilleszego_bams43}, Peyerimhoff~\cite{peyerimhoff_mmj13},
Runckel~\cite{runckel_tams69}, and Ki and Kim~\cite{kikim_dmj00}. 

Let us also point out that the methods developed in the papers cited above seem not to be readily applicable to the case of the Coulomb wave function treated here.
A reason for this can be that, unlike Lommel polynomials, no explicit expression for the coefficients of the orthogonal polynomials associated with the Coulomb 
wave function is known. Rather than that, only a recurrence relation for these coefficients is available, 
see~\cite[Prop.~10]{stampachstovicek_jmaa14}.

The main results of this paper are Theorems~\ref{thm:det_hankel_coulomb} and~\ref{thm:coulomb_zeros}, and the paper is organized as follows. 
In Section~\ref{sec:coulomb_func}, we briefly recall necessary definitions and basic properties of the regular Coulomb wave 
function and an associated spectral zeta function which, in a special case, simplifies to the well-known Rayleigh function. 

In Section~\ref{sec:hankel_det}, we prove a closed formula for the determinant of the Hankel matrix whose entries are 
given by the zeta function associated with the regular Coulomb wave function. As a corollary, we obtain formulas for 
determinants of two Hankel matrices with entries given by the Rayleigh function. These two Hankel determinants
can be viewed as two particular instances of a more general Hankel matrix. Although no simple formula for the determinant 
of this more general matrix is expected, we provide some partial results on the account of its evaluation.
Yet another corollary on some Hankel determinants with Bernoulli and Genocchi numbers is presented. 

Finally, as an application of the main result of Section~\ref{sec:hankel_det} and the results of Grommer and Chebotarev, 
we prove the Hurwitz-type theorem for the zeros of the regular Coulomb wave function in Section~\ref{sec:zeros_coulomb}.

\section{The regular Coulomb wave function and the associated zeta function}\label{sec:coulomb_func}

Recall that regular and irregular Coulomb wave functions,
$F_{L}(\eta,\rho)$ and $G_{L}(\eta,\rho)$, are two linearly
independent solutions of the second-order differential equation
\[
  \frac{d^{2}u}{d\rho^{2}}+\left(1-\frac{2\eta}{\rho}
    - \frac{L(L+1)}{\rho^{2}}\right)\!u = 0,
\]
see, for instance, \cite[Chp.~14]{abramowitz64}. The function $F_{L}(\eta,\rho)$ admits the decomposition
\cite[Eqs.~14.1.3 and 14.1.7]{abramowitz64}
\begin{equation}
  F_{L}(\eta,\rho)
  = C_{L}(\eta)\rho^{L+1}\phi_{L}(\eta,\rho),
  \label{eq:F_L_decomp}
\end{equation}
where
\[
C_{L}(\eta):=\frac{2^{L}{\rm e}^{-\pi\eta/2}\,|\Gamma(L+1+i\eta)|}{\Gamma(2L+2)},
\]
and
\begin{equation}
  \phi_{L}(\eta,\rho) := {\rm e}^{-i\rho}\,{}_{1}F_{1}(L+1-i\eta;2L+2;2i\rho).
  \label{eq:phi_rel_1F1}
\end{equation}
The confluent hypergeometric function $\,{}_{1}F_{1}$ is defined by the power series~\cite[Chp.~13]{abramowitz64}
\[
 {}_{1}F_{1}(a;b;z):=\sum_{n=0}^{\infty}\frac{(a)_{n}}{(b)_{n}}\frac{z^{n}}{n!},
\]
for $a,b,z\in\C$, such that $b\notin-\N_{0}$, where the Pochhammer symbol $(a)_{0}=1$ and $(a)_{n}=a(a+1)\dots(a+n-1)$, for $n\in\N$.
Here and below, $\N$ is the set of all positive integers and $\N_{0}:=\{0\}\cup\N$.

For the particular values of parameters, $L=\nu-1/2$ and $\eta=0$, one
gets \cite[Eqs. 14.6.6 and 13.6.1]{abramowitz64}
\[
  F_{\nu-1/2}(0,\rho)  =  \sqrt{\frac{\pi\rho}{2}}\, J_{\nu}(\rho)
\]
and
\begin{equation}
  \phi_{\nu-1/2}(0,\rho)  =  {\rm e}^{-i\rho}\,
  {}_{1}F_{1}(\nu+1/2;2\nu+1;2i\rho)
  \,=\,\Gamma(\nu+1)\left(\frac{2}{\rho}\right)^{\!\nu}
  J_{\nu}(\rho),
\label{eq:phi_part}
\end{equation}
where $J_{\nu}$ is the Bessel function of the first kind. From this point of view, the regular Coulomb wave function represents a one-parameter generalization of the Bessel function of the first kind.

One can see from~\eqref{eq:F_L_decomp} that, with the possible exception of the origin, the zeros of $F_{L}(\eta,\cdot)$ are the same as of $\phi_{L}(\eta,\cdot)$.
If $L\notin-(\N+1)/2$, the function $\phi_{L}(\eta,\cdot)$ is well-defined for all $\eta\in\C$. Even for $L\in-(\N+1)/2$, the function $\phi_{L}(\eta,\cdot)$ is well-defined, 
if we additionally require $\ii\eta\in\Z$ and $L+2+\ii\eta\leq0$, in which case the confluent hypergeometric series in~\eqref{eq:phi_rel_1F1} is terminating. In general, 
the function $F_{L}(\eta,\rho)$ can be continued analytically to the complex values of all variables $L$, $\eta$, and $\rho$; the interested reader is referred to~\cite{dzieciol-etal_jmp99,humblet_ap84,barnett-thompson_jcp86}.

As one can readily verify by using the Taylor coefficients of the confluent hypergeometric function in~\eqref{eq:phi_rel_1F1}, $\phi_{L}(\eta,\cdot)$ is an entire function of order $1$ for $L\notin-(\N+1)/2$ and $\eta\in\C$.
Consequently, the series
\[
 \zeta_{L}(k):=\sum_{n=1}^{\infty}\frac{1}{\rho_{L,n}^{k}},
\]
where $\rho_{L,1},\rho_{L,2},\dots$ are the zeros of $\phi_{L}(\eta,\cdot)$, is absolutely convergent for $k\geq2$, $L\notin-(\N+1)/2$, and $\eta\in\C$. Here we use the notation 
from~\cite{stampachstovicek_jmaa14} where $\zeta_{L}$ is referred to as the spectral zeta function since the zeros of $\phi_{L}(\eta,\cdot)$ are eigenvalues of a certain Jacobi 
operator. Although $\zeta_{L}$ as well as $\rho_{L,n}$ depend also on $\eta$, the dependence is not explicitly designated for brevity. 

Let us remark that $\zeta_{L}(k)$ is a polynomial in $\eta$ and a rational function in $L$ with singularities in the set $-(\N+1)/2$, as it can be seen from the recurrence relation
\begin{equation}
  \zeta_{L}(2) = \frac{1}{2L+3}\left(1+\frac{\eta^{2}}{(L+1)^{2}}\right)
  \label{eq:zeta_F_2}
\end{equation}
and
\begin{equation}
  \zeta_{L}(k+1) = \frac{1}{2L+k+2}\left(\frac{2\eta}{L+1}\,\zeta_{L}(k)+\sum_{l=1}^{k-2}\zeta_{L}(l+1)\zeta_{L}(k-l)\right)\!,\ \ k\geq2,
  \label{eq:zeta_F_genrecur}
\end{equation}
see \cite[Eqs.~(78) and~(79)]{stampachstovicek_jmaa14} or \cite[Eq.~(2.18)]{baricz_jmaa15}.

In the special case of $L=\nu-1/2$ and $\eta=0$, one has
\begin{equation}
\zeta_{\nu-1/2}(k)=\sum_{n\in\Z\setminus\{0\}}\frac{1}{j_{\nu,n}^{k}},
\label{eq:zeta_bessel_zeros}
\end{equation}
where $j_{\nu,1}, j_{\nu,2},\dots$ are the zeros of $J_{\nu}$ for which either $\Re j_{\nu,n}>0$ or $\Im j_{\nu,n}>0$, if $\Re j_{\nu,n}=0$.
The remaining zeros are $j_{\nu,-n}:=-j_{\nu,n}$, for $n\in\N$, which follows from the fact that 
that the function $\rho\mapsto\rho^{-\nu}J_{\nu}(\rho)$ is even and the origin is not a zero.
Consequently, for $k\in\N$, one gets
\begin{equation}
\zeta_{\nu-1/2}(2k+1)=0
\label{eq:zeta_vanish}
\end{equation}
and
\begin{equation}
\zeta_{\nu-1/2}(2k)=2\sigma_{2k}(\nu)=2\sum_{n=1}^{\infty}\frac{1}{j_{\nu,n}^{2k}},
\label{eq:zeta_rayleigh_spec}
\end{equation}
where $\sigma_{2k}(\nu)$ is the Rayleigh function of order $2k$ introduced by Kishore in~\cite{kishore_pams63}.

\section{Hankel determinants}\label{sec:hankel_det}

For $L\notin-(\N+1)/2$, $\eta\in\C$, and $n\in\N$, we define the Hankel matrix
\begin{equation}
  H_{n}(L,\eta):=\begin{pmatrix}
	\zeta_{L}(2) & \zeta_{L}(3) & \dots & \zeta_{L}(n+1)\\
	\zeta_{L}(3) & \zeta_{L}(4) & \dots & \zeta_{L}(n+2)\\
	\vdots & \vdots & \ddots & \vdots\\
	\zeta_{L}(n+1) & \zeta_{L}(n+2) & \dots & \zeta_{L}(2n)
      \end{pmatrix}\!.
\label{eq:def_hankel_mat_coulomb}
\end{equation}
By making use of the properties of a particular family of orthogonal polynomials associated with the regular Coulomb wave function studied in~\cite{stampachstovicek_jmaa14}, we 
may deduce a simple formula for the determinant of $H_{n}(L,\eta)$.

\begin{thm}\label{thm:det_hankel_coulomb}
 For $L\notin-(\N+1)/2$, $\eta\in\C$, and $n\in\N$, one has 
 \begin{equation}
  \det H_{n}(L,\eta)=\prod_{k=0}^{n-1}\frac{1}{(2L+2n-2k+1)^{2k+1}}\left(1+\frac{\eta^{2}}{(L+n-k)^{2}}\right)^{\!k+1}.
 \label{eq:det_hankel_coulomb}
 \end{equation}
\end{thm}

\begin{proof}
 The proof is based on the well-known relation between the recurrence coefficients from the three-term recurrence for a family of orthogonal polynomials and the 
 determinant of the corresponding moment Hankel matrix. Namely, we use that for the orthogonal polynomials generated by the recurrence
 \begin{equation}
  p_{n+1}(z)=(z-b_{n})p_{n}(z)-a_{n}p_{n-1}(z), \quad n\in\N,
 \label{eq:recur_monic_ogp}
 \end{equation}
 with the initial conditions $p_{-1}(z)=0$ and $p_{0}(z)=1$, it holds
 \begin{equation}
  \Delta_{n}=\prod_{m=1}^{n-1}\prod_{j=1}^{m}a_{j}, \quad n\in\N,
 \label{eq:delta_n_rel_a_n}
 \end{equation}
 where
 \[
 \Delta_{0}:=1, \quad \Delta_{n}:=\det\left[\mathcal{L}\left(z^{i+j}\right)\right]_{i,j=0}^{n-1}, \quad n\in\N,
 \]
 and $\mathcal{L}$ is the corresponding normalized moment functional; see~\cite[Chp.~I, Thm.~4.2(a)]{chihara78}.
 
 To obtain the formula from the statement, we make use of the particular family of orthogonal polynomials studied in~\cite{stampachstovicek_jmaa14},
 for which the moment Hankel matrix coincides with $H_{n}(L,\eta)$ up to an unimportant multiplicative factor. These polynomials satisfies~\eqref{eq:recur_monic_ogp}
 with 
 \begin{equation}
  a_{n} = \frac{(n+L+1)^{2}+\eta^{2}}{(n+L+1)^{2}(2n+2L+1)(2n+2L+3)},
  \quad b_{n}=-\frac{\eta}{(n+L+1)(n+L+2)},
  \label{eq:lambda_w_coulomb}
 \end{equation}
 and the corresponding moment sequence is given by
 \[
  \mathcal{L}\left(z^{n}\right)=\frac{\zeta_{L}(n+2)}{\zeta_{L}(2)}, \quad n\in\N_{0},
 \]
 see \cite[Rem.~19]{stampachstovicek_jmaa14}. Consequently, by taking~\eqref{eq:delta_n_rel_a_n} into account, one gets
 \begin{equation}
  \det H_{n}(L,\eta)=\left(\zeta_{L}(2)\right)^{n}\prod_{m=1}^{n-1}\prod_{j=1}^{m}a_{j}, \quad n\in\N.
  \label{eq:det_H_n_rel_moment_det}
 \end{equation}
 Now, it suffices to use~\eqref{eq:zeta_F_2} and~\eqref{eq:lambda_w_coulomb} to deduce the formula from the statement. 
\end{proof}

\begin{rem}
 The functions $\zeta_{L}(n)$, for $n\geq 2$, appear as coefficients in the power series expansion of the logarithmic derivative of $\phi_{L}(\eta,\rho)$ with respect to $\rho$,
 see, for example, the first unlabeled equation above \cite[Eq.~(77)]{stampachstovicek_jmaa14}. 
 As a consequence, Theorem~\ref{thm:det_hankel_coulomb} could be equivalently proved by using the connection between the Hankel determinant~\eqref{eq:def_hankel_mat_coulomb} and 
 the known coefficients from the continued fraction expansion of the logarithmic derivative of $\phi_{L}(\eta,\rho)$, see \cite[Thm.~7.14]{jones_thron-80} and~\cite{barnett-etal_cpc74}.
 On the other hand, one may observe by comparing the formula in \cite[Eq.~7.2.18a]{jones_thron-80} with~\cite[Chp.~I, Thm.~4.2(a)]{chihara78}
 that the coefficient $k_{1}$ from the continued fraction expansion in \cite[Eq.~7.2.16]{jones_thron-80} actually coincides with $\zeta_{L}(2)$ and
 $k_{n}$ coincides with $a_{n-1}$ from \eqref{eq:lambda_w_coulomb} for $n\geq2$. By using \cite[Eq.~7.2.29]{jones_thron-80}, one arrives 
 at~\eqref{eq:det_H_n_rel_moment_det} again.
\end{rem}

As a corollary of Theorem~\ref{thm:det_hankel_coulomb}, we compute determinants of two Hankel matrices with entries given by the Rayleigh function of even order. For this purpose, we define
\[
 H_{n}^{(\ell)}(\nu):=\begin{pmatrix}
	\sigma_{2\ell+2}(\nu) & \sigma_{2\ell+4}(\nu) & \dots & \sigma_{2n+2\ell}(\nu)\\
	\sigma_{2\ell+4}(\nu) & \sigma_{2\ell+6}(\nu) & \dots & \sigma_{2n+2\ell+2}(\nu)\\
	\vdots & \vdots & \ddots & \vdots\\
	\sigma_{2n+2\ell}(\nu) & \sigma_{2n+2\ell+2}(\nu) & \dots & \sigma_{4n+2\ell-2}(\nu)
      \end{pmatrix}\!,
\]
for $n\in\N$, $\ell\in\N_{0}$, and $\nu\notin-\N$. 

\begin{cor}\label{cor:det_hankel_bessel}
 For $n\in\N$, $\nu\notin-\N$, and $\ell\in\{0,1\}$, one has
 \begin{equation}
  \det H_{n}^{(\ell)}(\nu)=2^{-2n(n+\ell)}\prod_{k=1}^{2n+\ell-1}\left(\nu+k\right)^{k-2n-\ell}.
 \label{eq:det_hankel_bessel}
 \end{equation}
\end{cor}

\begin{proof}
 By taking the particular parameters $L=\nu-1/2$ and $\eta=0$ in~\eqref{eq:def_hankel_mat_coulomb} and using~\eqref{eq:zeta_bessel_zeros}, \eqref{eq:zeta_vanish}, and~\eqref{eq:zeta_rayleigh_spec},
 one sees that the $(i,j)$-th element of the matrix $H_{n}(\nu-1/2,0)$ coincides with $2\sigma_{i+j}(\nu)$, if the parity of $i$ and $j$ is the same,
 while it vanishes, if the parity of $i$ and $j$ is different. The latter observation and simple manipulations with the determinants, see for 
 example~\cite[Lem.~1.34]{holtz-tyaglov_siam12}, yield the identities
 \begin{equation}
  \det H_{2n+1}(\nu-1/2,0)=2^{2n+1}\det H_{n+1}^{(0)}(\nu)\det H_{n}^{(1)}(\nu),
  \label{eq:det_col_det_bes_odd}
 \end{equation}
 and 
 \begin{equation}
  \det H_{2n}(\nu-1/2,0)=2^{2n}\det H_{n}^{(0)}(\nu)\det H_{n}^{(1)}(\nu),
  \label{eq:det_col_det_bes_even}
 \end{equation}
 for $n\in\N$. 
 
 From~\eqref{eq:det_col_det_bes_odd} and~\eqref{eq:det_col_det_bes_even}, one deduces that
 \[
  \det H_{n}^{(0)}(\nu)=\frac{\det H_{2n-1}(\nu-1/2,0)}{2\det H_{2n-2}(\nu-1/2,0)}\det H_{n-1}^{(0)}(\nu)
 \]
 and
 \[
  \det H_{n}^{(1)}(\nu)=\frac{\det H_{2n}(\nu-1/2,0)}{2\det H_{2n-1}(\nu-1/2,0)}\det H_{n-1}^{(1)}(\nu),
 \]
 for $n\geq2$, which further implies that
 \begin{equation}
  \det H_{n}^{(0)}(\nu)=\sigma_{2}(\nu)\prod_{k=2}^{n}\frac{\det H_{2k-1}(\nu-1/2,0)}{2\det H_{2k-2}(\nu-1/2,0)}
 \label{det_bes_0_inproof}
 \end{equation}
 and
 \[
  \det H_{n}^{(1)}(\nu)=\sigma_{4}(\nu)\prod_{k=2}^{n}\frac{\det H_{2k}(\nu-1/2,0)}{2\det H_{2k-1}(\nu-1/2,0)},
 \]
 for $n\in\N$. At this point, one can use Theorem~\ref{thm:det_hankel_coulomb} and the
 well-known formulas
 \begin{equation}
  \sigma_{2}(\nu)=\frac{1}{4(\nu+1)} \quad \mbox{ and } \quad \sigma_{4}(\nu)=\frac{1}{16(\nu+1)^{2}(\nu+2)},
 \label{eq:sig2_sig4}
 \end{equation}
 which can be computed, for example, from~\eqref{eq:zeta_F_2} and~\eqref{eq:zeta_F_genrecur}, to verify the formulas from the statement.
 
 Alternatively, one can use~\eqref{eq:det_H_n_rel_moment_det} which particularly reads
 \[
  \det H_{n}(\nu-1/2,0)=\left(2\sigma_{2}(\nu)\right)^{n}\Delta_{n}, \quad n\in\N,
 \]
 together with the identity
 \[
  \frac{\Delta_{n+1}}{\Delta_{n}}=\prod_{j=1}^{n}a_{j},
 \]
 as it follows from~\eqref{eq:delta_n_rel_a_n}, to rewrite the right-hand side of~\eqref{det_bes_0_inproof} getting
 \[
    \det H_{n}^{(0)}(\nu)=\left(\sigma_{2}(\nu)\right)^{n}\prod_{k=2}^{n}\prod_{j=1}^{2k-2}a_{j}, \quad n\in\N.
 \]
 Here $a_{j}$ is given by~\eqref{eq:lambda_w_coulomb} where $L=\nu-1/2$ and $\eta=0$, i.e.,
 \[
   \det H_{n}^{(0)}(\nu)=\left(\sigma_{2}(\nu)\right)^{n}\prod_{k=2}^{n}\prod_{j=1}^{2k-2}\frac{1}{4(\nu+j)(\nu+j+1)}, \quad n\in\N.
 \]
 Analogically, one shows that
  \[
   \det H_{n}^{(1)}(\nu)=\sigma_{4}(\nu)\left(\sigma_{2}(\nu)\right)^{n-1}\prod_{k=2}^{n}\prod_{j=1}^{2k-1}\frac{1}{4(\nu+j)(\nu+j+1)}, \quad n\in\N.
 \]
 Finally, recalling~\eqref{eq:sig2_sig4}, simple algebraic manipulations yield the identities from the statement.
\end{proof}

\begin{rem}
An incorrect formula for $\det H_{n}^{(0)}(\nu)$ appeared earlier in \cite[Thm.~1]{zhangchen_aaa14}. To obtain the correct formula,
one should write $\sigma_{\nu}^{(k+j-1)}$ instead of $\sigma_{\nu}^{(k+j+1)}$ and $2^{-(n+1)(2n+1)}$ instead of $2^{(n+1)(2n+1)}$
in the identity therein.
\end{rem}

\begin{rem}
 The identity~\eqref{eq:det_hankel_bessel} is no longer true if $\ell\geq2$. We do not have a formula for $\det H_{n}^{(\ell)}(\nu)$ for general $\ell\in\N_{0}$.
 Nevertheless, in principle, the determinant can be computed recursively for a fixed $\ell$. Indeed, with the aid of the Desnanot--Jacobi identity, one gets
 \begin{equation}
  \det H_{n+1}^{(\ell)}(\nu)\det H_{n-1}^{(\ell+2)}(\nu)=\det H_{n}^{(\ell+2)}(\nu)\det H_{n}^{(\ell)}(\nu)-\left(\det H_{n}^{(\ell+1)}(\nu)\right)^{\!2},
 \label{eq:desnanot-jacobi_hankel_det_bessel}
 \end{equation}
 for $\ell\in\N_{0}$ and $n\geq 2$. Observe that~\eqref{eq:desnanot-jacobi_hankel_det_bessel} is a first-order difference equation in~$n$ for $\det H_{n}^{(\ell+2)}(\nu)$,
 which can be readily solved. If we temporarily denote $d_{n}^{(\ell)}:=\det H_{n}^{(\ell)}(\nu)$ and ignore for a while a possible division by zero, then the solution of 
 the difference equation~\eqref{eq:desnanot-jacobi_hankel_det_bessel} is given by a somewhat cumbersome formula
 \[
  d_{n}^{(\ell+2)}=d_{n+1}^{(\ell)}\left(\frac{\sigma_{2\ell+6}(\nu)}{\sigma_{2\ell+2}(\nu)\sigma_{2\ell+6}(\nu)-\sigma_{2\ell+4}^{2}(\nu)}
  +\sum_{k=2}^{n}\frac{\left(d_{k}^{(\ell+1)}\right)^{\!2}}{d_{k}^{(\ell)}d_{k+1}^{(\ell)}}\right)\!, \quad n\in\N.
 \]
 
 On the other hand, it is not difficult to use~\eqref{eq:desnanot-jacobi_hankel_det_bessel} and Corollary~\ref{cor:det_hankel_bessel} in order to verify that the Hankel determinants for $\ell=2,3$ read
 \begin{equation}
  \det H_{n}^{(2)}(\nu)=2^{-2n(n+2)}(n+1)(n+\nu+1)\prod_{k=1}^{2n+1}\left(\nu+k\right)^{k-2n-2}
 \label{eq:det_H_n_2_nu}
 \end{equation}
 and
 \begin{align*}
  \det H_{n}^{(3)}(\nu)&=2^{-2n(n+3)}\prod_{k=1}^{2n+2}\left(\nu+k\right)^{k-2n-3}\\
  &\times\frac{1}{6}(n+1)(n+2)(n+\nu+1)(n+\nu+2)\left[2n^{2}+6n+3+\nu(2n+3)\right]\!.
 \end{align*}
\end{rem}

\begin{rem}
 Besides the obvious positivity of $\det H_{n}^{(\ell)}(\nu)$ for $n\in\N$, $\ell\in\{0,1\}$, and $\nu>-1$, one can also use the formula~\eqref{eq:det_hankel_bessel} 
 to prove that $\det H_{n}^{(\ell)}(\nu)$, as function of~$\nu$, is completely monotone for $\nu>-1$, with $n\in\N$ and $\ell\in\{0,1\}$ fixed.
 This can be easily checked by using the fact that a product of completely monotone functions is a completely monotone function. Similarly, one can use~\eqref{eq:det_H_n_2_nu}
 to verify the complete monotonicity of $\det H_{n}^{(2)}(\nu)$ for $\nu>-1$, with $n\in\N$ fixed. It seems that the same holds true even for $\ell\geq3$, however, 
 the verification would require a more sophisticated analysis.
\end{rem}

At last, we formulate yet another corollary which is obtained by even more special choice of parameters taking $\eta=0$ and either $L=0$ or $L=-1$. Recall that $\sigma_{2n}(\pm 1/2)$ can be expressed
in terms of even values of Riemann zeta function, which, in its turn, can be evaluated with the aid of Bernoulli numbers. Namely, one has \cite[Eqs.~(4) and~(5)]{kishore_pams63}
\begin{equation}
 \sigma_{2n}\left(\frac{1}{2}\right)=\frac{\zeta(2n)}{\pi^{2n}}=(-1)^{n+1}\frac{2^{2n-1}}{(2n)!}B_{2n}, \quad n\in\N,
\label{eq:sigma_2n_1/2}
\end{equation}
and
\begin{equation}
 \sigma_{2n}\left(-\frac{1}{2}\right)=\left(2^{2n}-1\right)\frac{\zeta(2n)}{\pi^{2n}}=(-1)^{n}\frac{2^{2n-2}}{(2n)!}G_{2n}, \quad n\in\N,
\label{eq:sigma_2n_-1/2}
\end{equation}
where $G_{n}=2(1-2^{n})B_{n}$ are the Genocchi numbers. 

The special case of Corollary~\ref{cor:det_hankel_bessel} yields formulas for determinants of Hankel matrices 
with entries given by either $B_{2n+2\ell}/(2n+2\ell)!$ or $G_{2n+2\ell}/(2n+2\ell)!$ for $\ell\in\{0,1\}$.
Formulas for determinants of Hankel matrices with entries given just by $B_{2n+2\ell}$, for $\ell\in\{0,1\}$, 
can be found in~\cite[Eqs.~(3.59) and~(3.60)]{krattenthaler_slc99}. The determinant of the Hankel matrix with 
entries given by $G_{2n}$ can be deduced from~\cite[Eq.~(3.12)]{dumontzeng-am94}.

\begin{cor}
 For all $n\in\N$ and $\ell\in\{0,1\}$, one has
 \begin{equation}
  \det\left(\frac{B_{2j+2i+2\ell-2}}{(2j+2i+2\ell-2)!}\right)_{i,j=1}^{n}=(-1)^{n\ell}2^{-n(4n+4\ell-1)}\prod_{k=1}^{2n+\ell-1}\left(k+\frac{1}{2}\right)^{\!k-2n-\ell}
 \label{eq:hankel_det_bernoulli}
 \end{equation}
 and
 \begin{equation}
  \det\left(\frac{G_{2j+2i+2\ell-2}}{(2j+2i+2\ell-2)!}\right)_{i,j=1}^{n}=(-1)^{n(\ell+1)}2^{-n(4n+4\ell-2)}\prod_{k=1}^{2n+\ell-1}\left(k-\frac{1}{2}\right)^{\!k-2n-\ell}.
 \label{eq:hankel_det_genocchi}
 \end{equation}
\end{cor}

\begin{proof}
 First, observe that if $H_{n}$ and $\tilde{H}_{n}$ are two Hankel matrices with $(H_{n})_{i,j}=h_{i+j}$ and $(\tilde{H}_{n})_{i,j}=\tilde{h}_{i+j}$ such that 
 $h_{m}=\alpha^{m} \tilde{h}_{m}$, for some $\alpha\in\C$, then $H_{n}=D_{n}(\alpha)\tilde{H}_{n}D_{n}(\alpha)$, where $D_{n}(\alpha)=\diag(\alpha,\alpha^{2}\dots,\alpha^{n})$. 
 In particular, one has $\det H_{n}=\alpha^{n(n+1)}\det \tilde{H}_{n}$.
 
 By using the above observation together with~\eqref{eq:sigma_2n_1/2} and~\eqref{eq:sigma_2n_-1/2}, one obtains
 \[
  \det\left(\frac{B_{2j+2i+2\ell-2}}{(2j+2i+2\ell-2)!}\right)_{i,j=1}^{n}=(-1)^{n\ell}2^{-n(2n+2\ell-1)}\det H_{n}^{(\ell)}\left(\frac{1}{2}\right)
 \]
 and
 \[
  \det\left(\frac{G_{2j+2i+2\ell-2}}{(2j+2i+2\ell-2)!}\right)_{i,j=1}^{n}=(-1)^{n(\ell+1)}2^{-n(2n+2\ell-2)}\det H_{n}^{(\ell)}\left(-\frac{1}{2}\right)\!,
 \]
 for $n\in\N$. Now, it suffices to apply Corollary~\ref{cor:det_hankel_bessel}.
\end{proof}

\begin{rem}
 For $\ell=0$, the formula~\eqref{eq:hankel_det_bernoulli} is a correct version of the expression that appeared in~\cite[Cor.~2]{zhangchen_aaa14}. In addition,
 it shows that the determinant $\Delta_{m}'$ considered in~\cite[Ex.~3]{kytmanovkhodos_caop17} is indeed positive for all $m\in\N$.
\end{rem}

\begin{rem}
 It might be also of interest that both determinants~\eqref{eq:hankel_det_bernoulli} and~\eqref{eq:hankel_det_genocchi} are equal to a reciprocal value of an integer
 which is easy to check.
\end{rem}

\section{An application to the zeros of the regular Coulomb wave function}\label{sec:zeros_coulomb}

In \cite[Prop.~13]{stampachstovicek_jmaa14}, it was proved that the zeros of $\phi_{L}(\eta,\cdot)$ are all real if $-1\neq L>-3/2$ and $\eta\in\R\setminus\{0\}$ or $L>-3/2$ 
and $\eta=0$. This was also observed earlier by Ikebe~\cite{ikebe_mc75} for integer values of $L$ (which is an unnecessary restriction). As an application of 
Theorem~\ref{thm:det_hankel_coulomb} combined with general results due to Grommer~\cite{grommer_14} and Chebotarev~\cite{chebotarev_ma28}, we can provide an 
alternative proof of the above statement and, moreover, complement it by adding an information on the exact number of complex zeros of the regular Coulomb wave 
function. 

For this purpose, we recall the results of Grommer and Chebotarev in a special form adjusted to the situation
concerning the entire functions of order~$1$, which is the case of interest here. Chebotarev generalized the theorem of Grommer which, in its turn, is a
generalization of an analogous statement known for polynomials and attributed to Hermite~\cite{hermite_56}. For further research on this account, the reader may also 
consult~\cite{chebotarev-meiman_49,krein_62,krein-langer_mn77}. Some of these classical results were recently rediscovered by Kytmanov and Khodos~\cite{kytmanovkhodos_caop17} 
not mentioning the references~\cite{grommer_14,chebotarev_ma28}.

\begin{thm}\label{thm:grommer-chebotarev}
 Let $f$ be an entire function of order $1$ with real Taylor coefficients. Denote
 \[
  D_{-1}:=1 \quad \mbox{ and } \quad D_{n}:=\det\left(s_{i+j}\right)_{i,j=0}^{n-1}, \; \mbox{ for } n\in\N_{0},
 \]
 where
 \[
  s_{k}:=\sum_{j=1}^{\infty}\frac{1}{z_{j}^{k+2}}, \quad k\in\N_{0},
 \]
 and $z_{1},z_{2},\dots$ are the zeros of $f$. Then the following statements hold true.
	\begin{enumerate}[{\upshape i)}]
	\item $\mathrm{(Grommer)}$ All the zeros of $f$ are real if and only if $D_{n}>0$ for all $n\in\N_{0}$.
	\item $\mathrm{(Chebotarev)}$ If the sequence $\{D_{n-1}D_{n}\}_{n\in\N_{0}}$ contains exactly $m$ negative numbers, 
	 then the function $f$ has $m$ distinct pairs of complex conjugate zeros and an infinite number of real zeros.
	\end{enumerate}
\end{thm}

Now we are ready to prove the following statement.

\begin{thm}\label{thm:coulomb_zeros}
 Suppose $\eta,L\in\R$. The the following claims hold true.
	\begin{enumerate}[{\upshape i)}]
	\item If $-1\neq L>-3/2$ for $\eta\neq0$ and $L>-3/2$ for $\eta=0$, then all zeros of $F_{L}(\eta,\cdot)$ are real.
	\item If $L<-3/2$ and $L \notin-\N/2$ for $\eta\neq0$ and $L \notin-\N-1/2$ for $\eta=0$, then $F_{L}(\eta,\cdot)$ 
	has $\lfloor-L-1/2\rfloor$ distinct pairs of complex conjugate zeros and an infinite number of real zeros.
	\end{enumerate}
\end{thm}

\begin{proof}
Recall that the zeros of $F_{L}(\eta,\cdot)$ are the same as the zeros of $\phi_{L}(\eta,\rho)$ with the possible exception of the origin
as it follows from~\eqref{eq:F_L_decomp}. We apply Theorem~\ref{thm:grommer-chebotarev} to the function $f(\rho):=\phi_{L}(\eta,\rho)$ which is entire of order $1$.

It is by no means obvious from \eqref{eq:phi_rel_1F1} that the Taylor coefficients of $f$ are real for $\eta, L\in\R$.
To see that this is indeed the case, one can apply the Kummer transform \cite[Eq.~13.1.27]{abramowitz64}
\[
 {}_{1}F_{1}(a;b;z)=e^{z}\,_{1}F_{1}(b-a;b;-z), \quad a,b,z\in\C,\ b\notin-\N_{0}, 
\]
in~\eqref{eq:phi_rel_1F1}. This shows that $\overline{f(\rho)}=f(\overline{\rho})$, and hence the assumptions of Theorem~\ref{thm:grommer-chebotarev}
are fulfilled.

Denote $D_{n}:=\det H_{n+1}(L,\eta)$ for $n\in\N_{0}$. It is obvious from the identity~\eqref{eq:det_hankel_coulomb} that, for $-1\neq L>-3/2$ and $\eta\neq0$, 
$D_{n}>0$ for all $n\in\N_{0}$. If $\eta=0$, then by~\eqref{eq:det_col_det_bes_odd} and~\eqref{eq:det_col_det_bes_even}, $D_{n}$ is equal to the product
of the expressions given in~\eqref{eq:det_hankel_bessel} where $\nu=L+1/2$. It can be readily checked that both these expressions are positive for $L>-3/2$
and so the value $L=-1$ need not be excluded in this special case. In total, the claim (i) follows from the part (i) of Theorem~\ref{thm:grommer-chebotarev}.

Further, it follows from~\eqref{eq:det_hankel_coulomb} that
\begin{align}
 D_{n-1}D_{n}=\frac{1}{2L+2n+3}&\left(1+\frac{\eta^{2}}{(L+n+1)^{2}}\right)\nonumber\\
 &\times\prod_{k=0}^{n-1}\frac{1}{(2L+2n-2k+1)^{4k+4}}\left(1+\frac{\eta^{2}}{(L+n-k)^{2}}\right)^{\!2k+3},
 \label{eq:DD_eta_neq_0}
\end{align}
for $\eta\neq0$, $L\notin-(\N+1)/2$, and $n\in\N_{0}$. Similarly as above, in the particular case when $\eta=0$, $\phi_{L}(0,\rho)$ is to be understood as the Bessel function in the sense of~\eqref{eq:phi_part}, and
the negative integer values of $L$ need not be excluded. In this case, the formula~\eqref{eq:DD_eta_neq_0} takes the form
\begin{equation}
 D_{n-1}D_{n}=\frac{1}{2L+2n+3}\prod_{k=0}^{n-1}\frac{1}{(2L+2n-2k+1)^{4k+4}},
\label{eq:DD_eta_eq_0}
\end{equation}
and holds true for $L\notin-\N-1/2$ and $n\in\N_{0}$. Recall that, in~\eqref{eq:DD_eta_neq_0} and~\eqref{eq:DD_eta_eq_0}, $D_{-1}=1$ and the empty product is set $1$ by definition. 
In any case, it is clear from~\eqref{eq:DD_eta_neq_0} and~\eqref{eq:DD_eta_eq_0} that the sign of $D_{n-1}D_{n}$ equals the sign of the factor $2L+2n+3$.
Consequently, the number of negative elements in $\{D_{n-1}D_{n}\}_{n\in\N_{0}}$ coincides with the the number of elements of the set $\{n\in\N_{0} \mid 2L+2n+3<0\}$.
This observation together with the part (ii) of Theorem~\ref{thm:grommer-chebotarev} implies the claim~(ii) and the statement is proved.
\end{proof}

Apart from the claim on the purely imaginary zeros, the particular case of Theorem~\ref{thm:coulomb_zeros} with $\eta=0$ 
and $L=\nu-1/2$ implies Hurwitz's theorem about the zeros of the Bessel function of the first kind, which can be formulated as follows.

\begin{thm}[Hurwitz] \label{thm:hurwitz} Then following statements hold true.
	\begin{enumerate}[{\upshape i)}]
	\item If $\nu>-1$, then all zeros of $J_{\nu}$ are real.
	\item If $-2s-2<\nu<-2s-1$ for $s\in\mathbb{N}_0$, then $J_{\nu}$ has $4s+2$ complex zeros, of which two are purely imaginary.
	\item If $-2s-1<\nu<-2s$ for $s\in\mathbb{N}$, then $J_{\nu}$ has $4s$ complex zeros, of which none are purely imaginary.
	\end{enumerate}
\end{thm}

\begin{rem}
 The occurrence of a pair of purely imaginary zeros seems to be a special feature of the Bessel function.
 No similar phenomenon was observed in the general case of the Coulomb wave function. The latter statement, however, is based on numerical experiments only. 
 For instance, with the aid of Wolfram Mathematica 11, we found the following numerical values for the non-real zeros of $\phi_{L}(3/2,\cdot)$:
 \begin{align*}
  0.1500\pm \ii0.2520, &\quad\mbox{ for } L=-7/4,\\
  -0.2147\pm\ii0.8230,\ 0.5887\pm\ii0.4090, &\quad\mbox{ for } L=-11/4,\\
  -0.8719\pm\ii1.2916,\ 0.3538 \pm\ii1.2646,\ 1.1374\pm\ii0.5345, &\quad\mbox{ for } L=-15/4.
 \end{align*}
\end{rem}

\section*{Acknowledgments}
The research of \' Arp\'ad Baricz was supported by the STAR-UBB Advanced Fellowship-Intern of the Babe\c{s}-Bolyai 
University of Cluj-Napoca. This author is also grateful to Mourad E.~H.~Ismail and \'Agoston R\'oth for useful discussions.
The authors thank both referees for useful comments, in particular, for providing the original references for Theorem~\ref{thm:grommer-chebotarev}.

\end{document}